\documentclass[11pt]{article}

\usepackage[a4paper, hmargin=0.8in, top=1in, bottom=1.1in, footskip=0.6in]{geometry}

\usepackage{amsmath}
\usepackage{amssymb}
\usepackage{amsfonts}
\usepackage{latexsym}
\usepackage{amsthm}
\usepackage{amsxtra}
\usepackage{amscd}
\usepackage{mathrsfs}
\usepackage{mathscinet}
\usepackage{bm}

\usepackage{color}
\usepackage{cite}
\usepackage{hyperref}

\theoremstyle{plain} 
\newtheorem{theorem}{Theorem}
\newtheorem*{theorem*}{Theorem}

\newtheorem*{lemma*}{Lemma}
\newtheorem{corollary}[theorem]{Corollary}
\newtheorem*{corollary*}{Corollary}

\newtheorem*{proposition*}{Proposition}

\newtheorem*{definition*}{Definition}

\newtheorem*{conjecture*}{Conjecture}

\newtheorem*{example*}{Example}

\newtheorem*{remark*}{Remark}

\definecolor{darkred}{rgb}{0.9,0,0.3}
\definecolor{darkblue}{rgb}{0,0.3,0.9}

\usepackage{ifthen}
\def\comment#1{\ifthenelse{\isodd{\value{page}}}{\marginpar{\raggedright\scriptsize{\textcolor{darkred}{#1}}}}{\marginpar{\raggedleft\scriptsize{\textcolor{darkred}{#1}}}}}  

\newcommand{\He}{\text{Hess}}

\newcommand{\R}{\mathbb{R}}

\newcommand{\id}{\mspace{2mu}\mathrm{i}\mspace{-0.6mu}\mathrm{d}} 
\renewcommand{\leq}{\leqslant}
\renewcommand{\geq}{\geqslant}
\renewcommand{\epsilon}{\varepsilon}

\newcommand{\pB}[1]{\Bigl({#1}\Bigr)}
\newcommand{\pbb}[1]{\biggl({#1}\biggr)}

\newcommand{\nnb}{\nonumber\\}

\DeclareMathOperator{\var}{var}
\DeclareMathOperator{\cov}{cov}
\DeclareMathOperator{\ent}{ent}

\title{A very simple proof of the LSI for high temperature spin systems}

\begin{document}
\date{January 31, 2018}
\author{Roland Bauerschmidt\footnote{University of Cambridge, Statistical Laboratory, DPMMS. E-mail: {\tt rb812@cam.ac.uk}.} \and
  \and Thierry Bodineau\thanks{CMAP \'Ecole Polytechnique, CNRS, Universit\'e Paris-Saclay. E-mail: {\tt thierry.bodineau@polytechnique.edu}.}}
\maketitle
\begin{abstract}
  We present a very simple proof that 
  the $O(n)$ model satisfies a uniform logarithmic Sobolev inequality (LSI)
  if the positive definite coupling matrix has largest eigenvalue less than $n$.
  This condition applies in particular to the SK spin glass model at inverse temperature $\beta < 1/4$.
  It is the first result of rapid relaxation for the SK model and requires significant cancellations
  between the ferromagnetic and anti-ferromagnetic spin couplings that cannot be obtained
  by existing methods to prove Log-Sobolev inequalities.
  The proof also applies to more general bounded and unbounded spin systems.
  It uses a single step of zero range renormalisation
  and Bakry--Emery theory for the renormalised measure.
\end{abstract}

\section{Main result and proof}

To be concrete, consider the $O(n)$ model, though our method applies more generally.
Let $\Lambda$ be a finite set and $(M_{xy})_{x,y\in\Lambda}$ be a symmetric positive definite matrix.
The $O(n)$ model is the measure
\begin{equation}
  \nu(d\sigma)
  =
  \frac{1}{Z} 
  e^{-\frac12 (\sigma,M\sigma)}
  \prod_{x\in \Lambda} \mu(d\sigma_x),
  \quad
  (\sigma, M\sigma) = \sum_{x,y\in\Lambda} M_{xy} \, \sigma_x\cdot \sigma_y,
\end{equation}
where $\mu$ is the surface measure on $S^{n-1} \subset \R^n$.
The model is ferromagnetic if $M_{xy} \leq 0$ for $x \neq y$, but we do not require it.
For the $O(n)$ model, the diagonal entries of $M$ do not affect the measure,
so it is not a restriction to assume that $M$ is positive definite.
Denote by $\|M\|$ its maximal eigenvalue.

For any measure $m$, the entropy is $\ent_m(F) = m(F \log F) - m(F)\log m(F)$.
For $F: S^{n-1}\to \R$ in $C^1$, write $|\nabla F|^2$
for the square with respect to the metric of the covariant
derivative on $S^{n-1}$ if $n>1$ and $|\nabla F|^2 = |F(\sigma)-F(-\sigma)|^2$ if $n=1$.
Let $\mu^h(d\sigma) \propto e^{h\cdot \sigma} \mu(d\sigma)$
be the normalised single-spin measure
with external field $h \in \R^n$. Then $\mu^h$ satisfies
a LSI with constant $\gamma$ independent of $h$:
\begin{equation} \label{e:LSI-single}
  \ent_{\mu^h}(F^2) \leq \frac{2}{\gamma} \mu^h(|\nabla F|^2).
\end{equation}
For the Ising model ($n=1$),
the measure $\mu^h$ is Ber$(p)$ on $\{\pm 1\}$ with $p=\mu^h(1)$ and \eqref{e:LSI-single}
holds with $2/\gamma= 1/2 \geq pq(\log p-\log q)/(p-q)$ for any $p=1-q\in[0,1]$; see \cite{MR1849347,MR1490046,MR1845806}.
For $n > 2$, see \cite{MR2854732}; the  proof there can be adapted to $n=2$ (and actually becomes simpler).

\begin{theorem} \label{thm:LSI-On}
  Assume $M$ is positive definite and $\|M\| < n$.
  Then $\nu$ satisfies a LSI uniformly with respect to the set $\Lambda$:
  \begin{equation} 
  \label{e:LSI-On}
    \ent_\nu(F^2) \leq \frac{2}{\gamma}
    \left(1+\frac{2 n\|M\|}{n-\|M\|}\right)
    \sum_{x\in\Lambda} \nu(|\nabla_{\sigma_x} F|^2) .
  \end{equation}
\end{theorem}

General background on LSI and their consequences is summarised in
\cite{MR1767995,MR1837286,MR1971582,MR1490046,MR1746301,MR1845806}.
Several proofs of versions of the LSI are known
\cite{MR1156671,MR1152374,MR1153990,MR1182416,MR893137,MR1233852,MR1746301,MR1269388}.
Our main  
novelties are, firstly, the method of smoothing in the field instead of spatial decoupling 
and the resulting very simple proof, and, secondly, 
our spectral condition on $M$.
Note that, since the spins have norm $1$, adding multiple of the the identity matrix to $M$ does not change the measure $\nu$.
Thus if $M$ is not positive definite but has spectrum in $[\lambda^-,\lambda^+]$, 
then it is equivalent to apply the theorem with the coupling 
matrix $M$ replaced by $M-\lambda^- \id$ so that
the LSI holds under the spectral condition 
$\|M-\lambda^- \id\| = \lambda^+ - \lambda^- < n$.
This condition is effective even when significant cancellations between ferromagnetic
and anti-ferromagnetic couplings are required, and it thus applies to situations covered by none of the
existing methods for Log-Sobolev inequalities.

Most importantly, new applications include the Sherrington--Kirkpatrik (SK) spin glass model \cite{MR2731561}.
The coupling matrix of the SK model is $M=\beta H$ with  $\beta>0$ and $H$ a $N\times N$
GOE matrix consisting of independent Gaussian entries with variance $1/N$ above the diagonal.
Our theorem implies the following corollary.

\begin{corollary}
Let $\Gamma_N(M)$ be the Log-Sobolev constant associated with the quenched SK
measure on $N$ sites with coupling matrix $M=\beta H$.
Then the SK model with $\beta<1/4$ satisfies a uniform LSI in the following sense:
there is $c_\beta < \infty$ such that 
\begin{equation}
\lim_{N \to \infty} \mathbb{P}_N \big(  \Gamma_N(M) < c_\beta \big) = 1,
\end{equation}
where $\mathbb{P}_N$ stands for the GOE distribution of the coupling matrix.
\end{corollary}

\begin{proof}
Results on the concentration of the extreme eigenvalues of the GOE
imply that the distance $\lambda^+-\lambda^-$ between the smallest and largest eigenvalue of $H$ is concentrated at $4$   (see, e.g., \cite{MR2760897}). 
Thus the spectral condition implies the validity of the LSI with probability going to 1 as soon as
$4\beta < n=1$.
\end{proof}

For recent results on the dynamics of low temperature and spherical spin
glasses, see \cite{1608.06609,1705.04243} and references.
An exponentially large upper bound on the mixing time of the SK model
was given in \cite{MR1799876}.

For direct extensions of our proof to more general bounded and unbounded single-spin measures
satisfying a LSI uniformly in an external field, see
Section~\ref{sec:remarks}.

\begin{proof}
By possibly replacing the coupling $M$
by $M + \delta \id$ with $0< \delta < n - \|M\|$, we can assume that $0<M<c \id$ as quadratic forms on $\R^\Lambda$ with $c<n$.
Thus there is a positive definite matrix $B$ such that $M^{-1} = c^{-1} \id + B^{-1}$ and 
the Gaussian measure with covariance $M^{-1}$ can be represented as the  convolution of two Gaussian measures with covariance $c^{-1} \id $ and $B^{-1}$,
\begin{equation} 
\label{e:Gauss}
  e^{-(\sigma,M\sigma)/2}
  =
  C \int_{\R^{n\Lambda}} e^{-c(\varphi-\sigma,\varphi-\sigma)/2} e^{-(\varphi,B\varphi)/2} \, d\varphi
  \quad
  \text{for $\sigma \in \R^{n\Lambda}$.}
\end{equation}
For $\psi \in \R^n$, define the renormalised single-spin potential $V(\psi)$ and the probability measure $\mu_\psi$ by
\begin{equation} \label{e:mupsi}
  V(\psi) = -\log \int e^{-c(\psi-\sigma)^2/2} \, \mu(d\sigma),
  \qquad
  \mu_{\psi}(d\sigma) = e^{V(\psi)} e^{-c(\psi-\sigma)^2/2} \, \mu(d\sigma).
\end{equation}
Using that $|\sigma|=1$, notice that $\mu_\psi = \mu^h$ with $h=c\psi$ and where $\mu^h$ was defined above \eqref{e:LSI-single}.
For any $x\in \R^n$ with $|x|=1$ and any $h \in \R^n$, one has the bound $\var_{\mu^h}(x\cdot \sigma) \leq 1/n$.
This bound is trivial for $n=1$ and proved in \cite[Theorem~D.2]{MR0496246} for $n>1$
(alternatively one can assume $\|M\|<1$ and use the trivial bound $1$ for the variance also for $n>1$).
Thus
\begin{equation} \label{e:cov}
  x\cdot \He V(\psi)x
  = c|x|^2 - c^2 \var_{\mu_\psi}(x\cdot \sigma)
  \geq \lambda |x|^2, \quad \text{where $\lambda = c-c^2/n > 0$}
  .
\end{equation}
By the Bakry--Emery criterion \cite{MR889476}, it follows that the renormalised measure $\nu_r$ on $\R^{n\Lambda}$,
defined by
\begin{equation} \label{e:nur}
  \nu_r(d\varphi) = \frac{1}{Z_r} e^{-(\varphi,B\varphi)/2 - \sum_x V(\varphi_x)} d\varphi
  ,
\end{equation}
satisfies a LSI with constant $\lambda>0$.
For any $\psi \in \R^n$, the measure $\mu_\psi$ satisfies a LSI with constant $\gamma$
by~\eqref{e:LSI-single}. 
For any $\varphi \in \R^{n\Lambda}$, define 
$\mu_{\varphi}(d\sigma) = \prod_{x} \mu_{\varphi_x}(d\sigma_x)$. Then
$\nu(F(\sigma)) = \nu_r(\mu_\varphi(F(\sigma))$ and
by the tensorisation principle $\mu_{\varphi}$
satisfies a LSI with the same constant $\gamma$ as $\mu_{\varphi_x}$.
Let $G(\varphi) = \mu_{\varphi}(F(\sigma)^2)^{1/2}$. Using the LSI for $\mu_\varphi$ and $\nu_r$,
\begin{equation}
  \ent_\nu(F^2)
  = \nu_r(\ent_{\mu_{\varphi}}(F(\sigma)^2)) + \ent_{\nu_r}(G(\varphi)^2)
  \leq
  \frac{2}{\gamma} \sum_{x\in\Lambda} \nu(|\nabla_{\sigma_x}F(\sigma)|^2)
  + \frac{2}{\lambda} \sum_{x\in\Lambda} \nu_r(|\partial_{\varphi_x} G(\varphi)|^2)
  ,
\end{equation}
where $\partial_{\varphi_x}$ denotes the gradient in $\R^n$.
The following inequality completes the proof of \eqref{e:LSI-On}:
\begin{equation} \label{e:last}
  \nu_r(|\partial_{\varphi_x} G(\varphi)|^2)
  \leq
  \frac{2c^2}{\gamma}
  \nu_r(\mu_{\varphi}(|\nabla_{\sigma_x}F|^2))
  =
  \frac{2c^2}{\gamma}
  \nu(|\nabla_{\sigma_x}F|^2).
\end{equation}
This inequality follows from standard reasoning which goes as follows. By definition,
\begin{equation}
  \partial_{\varphi_x} G(\varphi)
  = \frac{\partial_{\varphi_x} G(\varphi)^2}{2G(\varphi)}
  = \frac{\partial_{\varphi_x} \mu_{\varphi}(F^2)}{2\mu_{\varphi}(F^2)^{1/2}}
  = \frac{c}{2} \frac{\cov_{\mu_{\varphi}}(F^2,\sigma_x)}{\mu_{\varphi}(F^2)^{1/2}}
  .
\end{equation}
By duplication of $\sigma_x$ (with the other spins fixed),
the Cauchy-Schwarz inequality, and $|\sigma_x-\sigma_x'|\leq 2$,
\begin{align} \label{e:F2}
  |\cov_{\mu_{\varphi_x}}(F(\sigma)^2,\sigma_x)|
  &= \frac12 |(\mu_{\varphi_x}\otimes\mu_{\varphi_x})\pB{(F(\sigma_x)-F(\sigma_x'))(F(\sigma_x)+F(\sigma_x'))(\sigma_x-\sigma_x')}|
  \nnb
  &\leq \pbb{\var_{\mu_{\varphi_x}}(F)}^{1/2} \pbb{ \frac12 (\mu_{\varphi_x}\otimes\mu_{\varphi_x})\pB{(F(\sigma_x)+F(\sigma_x'))^2|\sigma_x-\sigma_x'|^2}}^{1/2}
  \nnb
  &\leq \pB{\var_{\mu_{\varphi_x}}(F)}^{1/2} \pB{ 8\mu_{\varphi_x}(F(\sigma)^2)}^{1/2},
\end{align}
where the measure $\mu_{\varphi_x}$ fixes all spins except for $\sigma_x$.
Using $\mu_{\varphi}(\cdot) = \mu_{\varphi}(\mu_{\varphi_x}(\cdot))$ and independence, we have
$\cov_{\mu_{\varphi}}(F^2,\sigma_x) = \mu_{\varphi}(\cov_{\mu_{\varphi_x}}(F^2,\sigma_x))$,
so that by \eqref{e:F2} and the Cauchy--Schwarz inequality,
\begin{equation}
  |\cov_{\mu_{\varphi}}(F^2,\sigma_x)|^2
  \leq 8 \mu_\varphi(\var_{\mu_{\varphi_x}}(F)^{1/2}\mu_{\varphi_x}(F^2)^{1/2})^2
  \leq 8 \mu_\varphi(\var_{\mu_{\varphi_x}}(F)) \mu_{\varphi}(F^2)
  .
\end{equation}
By the spectral gap inequality for $\mu_{\varphi_x}$ implied by the LSI for $\mu_{\varphi_x}$,
\begin{equation}
  \mu_\varphi(\var_{\mu_{\varphi_x}}(F))
  \leq
  \frac{1}{\gamma} \mu_\varphi(\mu_{\varphi_x}(|\nabla_{\sigma_x}F|^2))
  = \frac{1}{\gamma} \mu_{\varphi}(|\nabla_{\sigma_x} F|^2)
  .
\end{equation}
In summary, we have shown
\begin{equation}
  |\partial_{\varphi_x} G(\varphi)|^2 \leq \frac{2c^2}{\gamma} \mu_{\varphi}(|\nabla_{\sigma_x}F|^2)
  .
\end{equation}
This implies \eqref{e:last} with the constant
$\frac{2}{\gamma} \left(1+\frac{2 n c}{n-c}\right)$.
The constant $c$ can be chosen as $\|M \| - \delta$ for any regularisation parameter $\delta>0$.
Thus letting $\delta>0$ tend to 0, this completes the proof.
\end{proof}

\section{Remarks and extensions}
\label{sec:remarks}

(i)
It is straightforward to adapt our proof to prove a spectral gap inequality under the assumption
that the single-spin measure satisfies a spectral gap inequality. It then becomes even simpler.

\smallskip\noindent
(ii)
Our proof applies without change to arbitrary 
(site-dependent) single-spin measures $\mu$ on $\R^n$
with support in the unit ball (and by rescaling with any bounded support)
such that the measure $\mu_\psi$ defined in \eqref{e:mupsi} satisfies a LSI with constant $\gamma$ uniformly in $\psi$.

\smallskip\noindent
(iii)
The proof can also be adapted easily to the case that $\mu$ has unbounded support.
In particular, assume that the single-spin measures $\mu$ has density proportional
to $e^{-U}$ with respect to the Lebesgue measure on $\R^n$,
with $U$ convex at infinity (perturbed convex).
By the Bakry--Emery and Holley--Stroock criteria (see \cite{MR1837286}),
$\mu_\psi$ then satisfies a LSI uniformly in the external field $\psi$.
In the proof of our theorem,
boundedness is used only in \eqref{e:cov} and \eqref{e:F2}.
In the above setting, a version of \eqref{e:cov} follows directly
from the fact that the LSI for $\mu_\varphi$ implies it has spectral gap $\gamma$
(or the variance can be estimated by hand for a possibly better estimate)
and \eqref{e:F2} can be replaced
by the generalisation \cite[Proposition~2.2]{MR1837286} due to \cite{MR1715549,MR1704666}.
Except for an additional constant in the second term in \eqref{e:LSI-On}, the conclusion is identical.

\smallskip\noindent
(iv)
For ferromagnetic spin couplings, the spectral condition on $M$ is sharp for the mean-field model.
Also, for positive definite spin coupling matrices $M$, the bound $\|M\| < n$ is
implied by the mean-field bound \cite{MR589428}:
\begin{equation} \label{e:MFcondition}
\sup_x \sum_{y \in \Lambda} |M_{xy}| < n.
\end{equation}
In general, for spin coupling matrices $M$ that are not positive definite,
the shifting of the spectrum to make them positive definite costs at most a factor two.
However, as exemplified by the 
SK model, our spectral condition does not involve an absolute value of $M$ and is much stronger
when cancellations between spin couplings are relevant as in the SK model.

\smallskip
\noindent
(v)
We view our starting point \eqref{e:Gauss} as a single step of renormalisation.
It is similar but not identical to the Hubbard--Stratonovich transform
used in \cite{BraLieAppl,MR0496246}
to obtain a bound on the two-point function of the $O(n)$ model
under a similar assumption. The Hubbard--Stratonovich transform is
\begin{equation}
  e^{+(\sigma,M\sigma)/2} = C \int_{\R^{n\Lambda}} e^{-(\varphi,\sigma) - (\varphi,M^{-1}\varphi)/2} \, d\varphi.
\end{equation}
Its utility is emphasised in \cite{MR2953867}.

\section*{Acknowledgements}

We thank
F.\ Barthe, 
G.\ Ben Arous,
N.\ Berestycki,
D.\ Brydges,
D.\ Chafa\"i,
A.\ Jagannath,
J.\ Hermon,
T.\ Spencer,
and Y.\ Spinka
for discussions and references.
 We acknowledge the support of ANR-15-CE40-0020-01 grant LSD.

\bibliography{all}
\bibliographystyle{plain}

\end{document}